\documentclass[a4paper,reqno,notitlepage]{amsart}
\usepackage{amsrefs}
\usepackage{amsfonts}
\usepackage{amscd,amsmath}
\usepackage{amssymb}
\usepackage{euscript}
\usepackage{mathrsfs}
\usepackage{tikz-cd}

\makeatletter
\pgfqkeys{/tikz/commutative diagrams}{
  row sep/.code={\tikzcd@sep{row}{#1}{}},
  column sep/.code={\tikzcd@sep{column}{#1}{}},
  bo row sep/.code={\tikzcd@sep{row}{#1}{between origins}},
  bo column sep/.code={\tikzcd@sep{column}{#1}{between origins}},
  bo column sep/.default=normal,
  bo row sep/.default=normal,
}
\def\tikzcd@sep#1#2#3{
  \pgfkeysifdefined{/tikz/commutative diagrams/#1 sep/#2}%
    {\pgfkeysalso{/tikz/#1 sep={\ifx\\#3\\1*\else1.7*\fi\pgfkeysvalueof{/tikz/commutative diagrams/#1 sep/#2},#3}}}%
    {\pgfkeysalso{/tikz/#1 sep={#2,#3}}}}
\makeatother

\newtheorem{theorem}{Theorem}[section]

\newtheorem{prop}[theorem]{Proposition}

\newtheorem{cor}[theorem]{Corollary}

\theoremstyle{definition}
\newtheorem{note}[theorem]{Note}

\newtheorem{problem}[theorem]{Problem}
\newtheorem{example}[theorem]{Example}

\newcommand{\N}{\mathbb N}
\newcommand{\Z}{\mathbb Z}

\newcommand\rr{{\mathfrak r}}
\renewcommand\ss{{\mathfrak s}}
\newcommand\K{\Bbbk}
\newcommand\U{\mathscr U}
\newcommand\gr{\operatorname{gr}}

\begin{document}
\title{Lie Dimension Subrings}
\author{Laurent Bartholdi}
\address{\begin{flushleft}
    \upshape Laurent Bartholdi\\
    Mathematisches Institut\\
    Georg-August Universit\"{a}t zu G\"ottingen\\
    G\"{o}ttingen, Germany\\
    E-mail: \texttt{laurent.bartholdi@gmail.com}
\end{flushleft}}

\author{Inder Bir S. Passi}
\address{\begin{flushleft}
    \upshape Inder Bir S. Passi\\
    Centre for Advanced Study in Mathematics\\
    Panjab University\\
    Chandigarh, India\\
    and\\
    Indian Institute of Science Education and Research\\
    Mohali, India\\
    Email: \texttt{ibspassi@yahoo.co.in}
\end{flushleft}}

\date{August 12, 2013}



\begin{abstract}
  We compare, for $L$ a Lie ring over the integers, its lower central
  series $(\gamma_n(L))_{n\ge1}$ and its \emph{dimension series}
  defined by $\delta_n(L):=L\cap \varpi^n(L)$ in the universal
  enveloping algebra of $L$. We show that $\gamma_n(L)=\delta_n(L)$
  for all $n\le3$, but give an example showing that they may differ if
  $n=4$. We introduce simplicial methods to describe these results,
  and to serve as a possible tool for further study of the dimension
  series.
\end{abstract}
\maketitle

\section{Introduction}
Let $G$ be a group and let $\Z[G]$ be its integral group ring with
augmentation ideal $\varpi(G)$. Two series of subgroups may be
considered: the first, purely group-theoretical, is the \emph{lower
  central series} $\gamma_n(G)$ defined inductively by setting
$\gamma_1(G)=G$ and $\gamma_{n+1}(G)=[\gamma_n(G),G]$. The second,
more algebraic, is the \emph{dimension series}
$\delta_n(G)=G\cap(1+\varpi(G)^n)$. One always has
$\gamma_n(G)\le\delta_n(G)$, and in fact the quotient
$\delta_n(G)/\gamma_n(G)$ is a torsion abelian
group~\cite{Gupta-Kuzmin:92} whose exponent is bounded by a function
of $n$ only~\cite{Sjogren:79}. However, its precise structure, despite
extensive investigation, is a still not completely understood question
(see~\cites{Gupta:87,MP:09,Passi:79}) known as the \emph{dimension
  subgroup problem}.

The aim of the present paper is to study an analogous problem for Lie
rings and their universal enveloping algebras. On the one hand, the
arguments in the case of groups have a strong Lie-theoretical
flavour, so it seems desirable to cast them in their natural
environment. On the other hand, there is a classical construction of a
Lie ring (over $\Z$) out of a group, due to Magnus~\cite{magnus:lie},
see also~\cite{labute:magnus}. As an abelian group, it is the direct
sum $\gr(G)$ of successive quotients $\gamma_n(G)/\gamma_{n+1}(G)$;
the Lie bracket comes from the group commutator on homogeneous
elements. We attempt to establish a different link between groups and
Lie algebras. Finally, we believe that Lie algebras are important
objects to study in their own right.

If $L$ be a Lie ring over a commutative ring $\K$ with identity,
$\U(L)$ its universal enveloping algebra and $\varpi(L)$ the
augmentation ideal of $\U(L)$, then we have, for every integer $n\geq
1$, a Lie subring $\delta_n(L):= L\cap \varpi^n(L)$ of $L$, called the
$n$th \emph{Lie dimension subring} of $L$. Once again the $n$th lower
central Lie subring $\gamma_n(L)$ of $L$ is always contained in
$\delta_n(L)$, and there arises the problem of identifying the
quotient $\delta_n(L)/\gamma_n(L)$. While a complete answer to this
problem is known in case $\K$ is a field
(see~\cites{Knus:1969,Riley:1995}), the ``universal'' case $\K=\Z$
does not seem to have been investigated so far.

\subsection{Main results}
Let $L$ be a Lie ring, and consider the lower central series
$\{\gamma_n(L)\}_{n\ge 1}$ and dimension series
$\{\delta_n(L)\}_{n\geq1}$. The terms $\gamma_1(L)$ and $\delta_1(L)$
are by definition equal. We prove in Theorem~\ref{case n=2} that
$\gamma_2(L)=\delta_2(L)$, and in Theorem~\ref{thm:delta3} that
$\gamma_3(L)=\delta_3(L)$. We show by an example that they can differ
for $n=4$, see Theorem~\ref{case n=4}; in that case, nevertheless,
$2\delta_4(L)\subseteq \gamma_4(L)$, see Corollary \ref{2delta_4}.

Given a free presentation $0\to R\to F\to L\to 0$ of a Lie ring $R$,
one may ask, in analogy with the Fox subgroup
problem~\cites{Fox:1953,Gupta:87}, for the identification of the Lie
subrings $F\cap \varpi^n(F)\rr$, with $\rr$ the two-sided ideal
generated by $R$, in the universal enveloping algebra $\U(F)$ of
$F$. This too is going to be a challenging problem; for, a simple
example shows that, unlike in the case of groups, $(F\cap
\varpi(F)\rr)/[R,R]$ can be non-zero.

We eschew the problem of identifying $F\cap\varpi(F)\rr$, setting
$M=F\cap\varpi(F)\rr$, and derive some results relating $F/M$ to the
universal enveloping algebra $\U(F/R)$, motivated by their
group-theoretic counterpart. We show, under the assumption that $R/M$
has trivial annihilator in $\U(F/R)$, that
$\bigcap_{n\ge1}\varpi^n(F/R)=0$ if and only if
$\bigcap_{n\ge1}\gamma_n(F/M)=0$. The assumption on $R/M$ always holds
if $F$ is a Lie algebra over a field.

Finally, we develop simplicial methods, analogous to those
in~\cites{HMP:07,MP:09}, to investigate Lie dimension subrings. We
obtain in this manner a ``conceptual'' proof of
Theorem~\ref{thm:delta3}; but we also expect these methods to bear
more fruits in the future.

\section{Notation}
The following notation will be used throughout the text:
\renewcommand{\descriptionlabel}[1]{\hspace{\labelsep}$#1\,=$}
\begin{description}
\item[\K] a commutative ring with identity
\item[\N] the natural numbers $\{0,1,\dots\}$
\item[L] a Lie ring over $\K$
\item[\gamma_n(L)] the $n$th term in the lower central series of
  $L$. It is defined inductively by $\gamma_1(L)=L$ and, for $n\geq
  1$, by letting $\gamma_{n+1}(L)$ be the $\K$-submodule of $L$
  generated by all elements of the form $[x,y]=xy-yx$ with $x\in
  L,y\in\gamma_n(L)$.
\item[\varGamma_n(L)] the quotient $L/\gamma_n(L)$
\item[L''] the second derived Lie subring $[[L,L],[L,L]]$
\item[\mathscr T(A)] the tensor algebra over the $\K$-module $A$
\item[\mathscr L(A)] the free Lie algebra over the $\K$-module $A$; it
  is a subspace of $\mathscr T(A)$
\item[\operatorname{Sym}(A)] the quotient $\mathscr T(L)/\langle
  x\otimes y-y\otimes x\colon x,y\in L\rangle$, the symmetric algebra
  of $L$
\item[\U(L)] $\mathscr T(L)/\langle x\otimes y-y\otimes x-[x,y]\colon
  x,y\in L\rangle$, the universal envelope of $L$
\item[\iota:L\to\U(L)] the natural homomorphism
\item[\U_n(L)] the homogeneous component of degree $n$ in
  $\U(L)$
\item[\gr(L)] $\bigoplus_{i=1}^\infty \gamma_n(L)/\gamma_{n+1}(L)$, the associated graded Lie ring with
  \[[\tilde{x}_i,\tilde{x}_j]=[x_i,x_j]+\gamma_{i+j+1}(L),\] for
  $\tilde{x}_i=x_i+\gamma_{i+1}(L)$, $\tilde{x}_j=x_j+\gamma_{j+1}(L)$,
  $x_i\in \gamma_i(L)$, $x_j\in \gamma_{j}(L)$.
\item[gr_n(L)] $\gamma_n(L)/\gamma_{n+1}(L)$ for $n\geq 1$
\item[\varpi(L)] the augmentation ideal of $\U(L)$, namely, the two-sided ideal of $\U(L)$ generated by $\iota(L)$
\item[\gr(\U(L))] $\bigoplus_{n=0}^\infty
  \varpi^n(L)/\varpi^{n+1}(L)$, the associated graded ring of
  $\U(L)$ arising from its $\varpi(L)$-adic filtration
\item[\U(\gr(L))] the universal envelope of the graded Lie ring
  $\gr(L)$
\item[\delta_n(L)] $L\cap \varpi^n(L)$ for $n\ge1$, the $n$th Lie
  dimension subring of $L$.
\end{description}
For basic properties of Lie algebras, the reader may refer to the
classic~\cite{Jacobson:1962}.

\section{Lie lower central and dimension subrings}
We begin by listing some properties of the Lie lower central and
dimension subrings which are easily verified.
\begin{prop}
  For every $m,n\ge1$, we have
  \begin{enumerate}
  \item[(i)] $[\delta_{m}(L),\delta_n(L)]\subseteq \delta_{m+n}(L);$
  \item[(ii)] $\delta_{m}(L)$ is a Lie ideal of $L$;
  \item[(iii)] $\gamma_m(L)\subseteq \delta_{m}(L)$.\qed
  \end{enumerate}
\end{prop}

The homomorphism $\iota:L\to \U(L)$ is not, in general, a
monomorphism~\cites{Shirshov:PBW,Cartier:1958}; however, it is known
to be so in the following cases:
\begin{theorem}[Poincar\'{e}-Birkhoff-Witt; see~\cite{Higgins:69} for a unified proof]\label{thm:PBW}
  The Lie homomorphism
  $\iota:L\to \U(L)$ is injective if either
  \begin{enumerate}
  \item[(i)] $L$ is a free $\K$-module~\cites{Birkhoff:1937,Witt:37}, or
  \item[(ii)] $\K$ is a Dedekind domain~\cites{Cartier:1958,Lazard:1954}, or
  \item[(iii)] $\K$ is an algebra over the rationals~\cite{Cohn:1963}.
  \end{enumerate}

  In fact, in all these cases the stronger result holds, that the
  natural map $\operatorname{Sym}(L)\to\gr\U(L)$ is an isomorphism of
  $L$-modules.
\end{theorem}

\noindent There exists a canonical graded Lie ring homomorphism
\[\gr(\iota): \gr(L)\to \gr(\U(L)), \quad
\tilde{x}_n\mapsto x_n+\varpi^{n+1}(L)\text{ for }x_n\in
\gamma_n(L).
\]
Consequently, we have an induced $\K$-algebra homomorphism
\[\iota^*:\U(\gr(L))\to \gr(\U(L)).\]

\begin{theorem}[Riley~\cite{Riley:1995}; Knus~\cite{Knus:1969}]
  If $\K$ be a field of characteristic zero and $L$ be a Lie algebra
  over $\K$, then
  \begin{enumerate}
  \item[(i)] $\iota^*:\U(\gr(L))\to \gr(\U(L))$ is an isomorphism;
  \item[(ii)] $\delta_n(L)=\gamma_n(L)$ for all $n\geq 1$.
  \end{enumerate}
\end{theorem}
Similar results hold, with universal envelope (respectively Lie ring)
replaced by \emph{restricted universal envelope} (respectively
\emph{restricted Lie ring}), if $\K$ is a field of prime
characteristic; see~\cites{RS:1993,RS:1995a}. We are thus led to the
following

\begin{problem}
  If $L$ be a Lie ring over a commutative ring $\K$ with identity,
  identify the quotients $\delta_n(L)/\gamma_n(L)$, for all $n\geq 1$.
\end{problem}

In this paper we limit ourselves to the consideration of the case when
$\K$ is $\Z$, the ring of integers.

\section{The dimension problem for Lie rings over $\Z$}
We begin by considering free Lie rings over $\Z$. It may be recalled
that any subalgebra of a free Lie algebra over a field is itself free
(see~\cites{Shirshov:53,Witt:56}).  On the other hand, a subring of a
free Lie ring over $\Z$ is, in general, \emph{not} free;
see~\cite{Witt:56} for a counter-example. See~\cite{Reutenauer:93} for
an exhaustive treatise on free Lie algebras, and~\cite{Dixmier:77} as
a reference for universal enveloping algebras.

If $F$ be a free Lie ring over $\Z$, then
$\gamma_n(F)/\gamma_{n+1}(F)$ is a free abelian group
(see~\cite{Hall:50} for an explicit basis). Consequently, we have the
counterpart of the fundamental theorem of free group rings
(\cite{Gupta:87}*{Theorem 3.7}):

\begin{theorem}\label{thm:fund}
  If $F$ be a free Lie ring over $\Z$, then $F\cap
  \varpi^n(F)=\gamma_n(F)$ for all $n\geq 1$.
\end{theorem}
\begin{proof}
  Let $F$ be free on the subset $X$. The universal enveloping algebra
  $\U(F)$ is isomorphic to the algebra of non-commuting polynomials
  over $X$, and the Lie subalgebra of $\U(F)$ generated by $X$ is
  free, by Theorem~\ref{thm:PBW}(i). It is a graded subalgebra, so the
  algebras $F$, $\gr(F)=\bigoplus_{n\ge1}\gamma_n(F)/\gamma_{n+1}(F)$
  and $\bigoplus_{n\ge1}(F\cap\varpi^n(F))/(F\cap\varpi^{n+1}(F))$ are
  all isomorphic.
\end{proof}

We now fix a free presentation of the Lie ring $L$, namely, an exact
sequence
\[\begin{tikzcd}0\arrow{r} & R\arrow{r} & F\arrow{r}{p} & L\arrow{r} & 0\end{tikzcd}\]
of Lie rings with $F$ a free Lie ring and $R$ a Lie ideal in $F$. Then
we have the following
\begin{prop}[See~\cite{Dixmier:77}*{\S2.2.15}]\label{prop:uni}
  $\U(L)\simeq \U(F)/\rr$, where $\rr$ is the
  two-sided ideal generated by $R$ in $\U(F)$.
\end{prop}
\begin{proof}
  Recall the natural Lie homomorphism $\iota:L\to \U(L)$. The Lie
  homomorphism $\iota\circ p: F\to \U(L)$ induces a homomorphism
  $\theta:\U(F)\to \U(L)$ of associative algebras which clearly
  vanishes on the ideal $\rr$. Thus we have a homomorphism
  \[\overline{\theta}:\U(F)/\rr\to \U(L)\]
  of associative algebras.

  On the other hand, the map $F\to \U(F)/\rr$ defined by $f\mapsto
  f+\rr$ vanishes on $R$, and consequently induces a Lie homomorphism
  $\varphi:L\to \U(F)/\rr$. Therefore, we have a homomorphism
  \[\overline{\varphi}:\U(L)\to \U(F)/\rr\]
  which maps $w\in L$ to $f+\rr$ whenever $p(f)=w$.  Clearly
  $\overline{\theta}$ and $\overline{\varphi}$ are inverses of each
  other, so the proof is complete.
\end{proof}

Let $L$ be a Lie ring and consider $w\in L\cap\varpi^n(L)$ for some
$n\geq 1$. Choose $f\in F$ with $p(f)=w$. Then, in view of
Proposition~\ref{prop:uni}, we have $f\in \varpi^n(F)+\rr$. Observe that
\[\varpi^n(F)+\rr=\varpi^n(F)+\varpi(F)\rr+R.\]
Therefore, there exists $r\in R$ with $f+r\in
\varpi^n(F)+\varpi(F)\rr$. Consequently, in order to determine $L\cap
\varpi^n(L)$, it suffices to determine
\[F \cap (\varpi^n(F)+\varpi(F)\rr).
\]
Since $F\cap \varpi^2(F)=\gamma_2(F)$ by Theorem~\ref{thm:fund} and
$\rr\subseteq \varpi(F)$, we immediately have the following
\begin{theorem}\label{case n=2}
  For every Lie algebra $L$ over $\Z$, we have
  $\delta_2(L)=\gamma_2(L)$.\qed
\end{theorem}

\noindent The following result parallels
Gupta-Kuzmin's~\cite{Gupta-Kuzmin:92}.
\begin{prop}
  If $L$ be a Lie algebra over $\Z$, then
  $\delta_n(L)/\gamma_{n+1}(L)$ is abelian for all $n\ge1$.
\end{prop}
\begin{proof}
  Let $L$ be a Lie algebra over $\Z$, and assume that $L$ is nilpotent
  of class $n$, namely, $\gamma_{n+1}(L)=0$. Let $A$ be a maximal
  abelian ideal in $L$.

  We first show that $A$ equals its centralizer $Z_L(A)$. Indeed, if
  $Z_L(A)>A$, choose $\overline x\neq0$ in the centre of $Z_L(A)/A$,
  and let $x$ denote a lift to $Z_L(A)\setminus A$. Then $A+\Z x$ is a
  larger abelian ideal, contradicting the maximality of $A$. We
  naturally view $A$ as a right $\U(L)$-module, writing the action
  $[-,-]$. Then $[a,x]\in\gamma_{k+1}(L)$ for all $a\in A$,
  $x\in\varpi^k(L)$. In particular, every $x\in\delta_n(L)$ belongs to
  $Z_L(A)$, and therefore to $A$, so $\delta_n(L)$ is abelian.
\end{proof}

\subsection{\boldmath $n=3$}
We next proceed to examine the third and the fourth Lie dimension
subrings. Consider a Lie ring
\[L=\langle X_1,X_2,\ldots,X_m\mid e_1X_1+\xi_1,e_2X_2+\xi_2,\ldots,e_mX_m+\xi_m,\xi_{m+1},\ldots\rangle
\]
given by its \emph{preabelian} presentation. This is a presentation
making apparent the elementary divisors $e_1,\dots,e_m$ of $L/[L,L]$. Let $F$
be the free Lie ring generated by $X_1,\ldots,X_m$, and let $R$ be the
ideal of $F$ generated by
$e_1X_1+\xi_1,e_2X_2+\xi_2,\ldots,e_mX_m+\xi_m,\xi_{m+1},\ldots$,
where $e_1|e_2|\ldots|e_m$ are integers $\geq 0$ and
$\xi_1,\ldots,\xi_m\ldots,$ are certain elements of
$\gamma_2(F)$. Thus $L=F/R$. Let $\rr$ be the two-sided ideal
generated by $R$ in $\U(F)$, and let $\ss$ denote the two-sided ideal
of $\U(F)$ generated by $\{e_1X_1,\ldots,e_mX_m\}\cup\gamma_2(F)$;
thus $\ss=\rr+\U(F)\gamma_2(F)$. For notational brevity, we write
$\varpi$ for $\varpi(F)$.

\begin{theorem}\label{thm:delta3}
  For every Lie algebra $L$ over $\Z$, we have
  $\delta_3(L)=\gamma_3(L)$.
\end{theorem}
\begin{proof}
  Consider $w\in(\varpi^3+\mathfrak r)\cap F$ representing an element
  of $\delta_3(L)$. Then, for some $r\in R$, we have
  \[v:=w+r\in(\varpi^3+\varpi\rr)\cap F\subseteq(\varpi^3+\varpi\ss)\cap F\subseteq\gamma_2(F).
  \]
  We may therefore write
  \[v\equiv\sum_{i>j} a_{ij}[X_i,X_j]\mod \gamma_3(F).\]
  Since $v\in\varpi^3+\varpi\ss$, we also have
  \[v\equiv\sum_{i,j}c_{ij}e_jX_iX_j\mod \varpi^3.
  \]
  Equating co\"efficients of $X_iX_j$, we get
  $a_{ij}=c_{ij}e_j=-c_{ji}e_i$ for all $i>j$. Then, from $e_iX_i\in
  R+\gamma_2(F)$, we have
  \[\sum_{i>j}a_{ij}[X_i,X_j]\in [F,R]+\gamma_3(F),\]
  so $w\in \gamma_3(F)+R$.
\end{proof}

\noindent The preceding proof can be extended to yield the following
\begin{theorem}\label{thm:deltan}
  If $w\in(\varpi^n+\rr)\cap F$, then there exist simple
  commutators $c_1,\dots,c_\ell$, all of degree $\geq 2$, and
  co\"efficients $a_i\in\Z$, such that
  \[w\equiv\sum a_ic_i\mod\gamma_n(F)+F''+R,
  \]
  and, if $c_i=[X_{i_1},X_{i_2},\dots,X_{i_r}]$, then $a_i$ is
  divisible by $e_{i_1}$.
\end{theorem}
\begin{proof}
  For some $r\in R$, we have
  \[v:=w+r\in(\varpi^n+\varpi\rr)\cap F\subseteq 
  (\varpi^n+\varpi\ss)\cap F\subseteq\gamma_2(F).
  \]

  We view $F$ as a right $\U(F)$-module, for the adjoint action
  written $[-,-]$.  We claim that $F\cap(\varpi^n+\varpi\ss)$ is
  generated, modulo $\gamma_n(F)+F''$, by the elements
  $e_i[[X_i,X_j],u_{ij}]$ for all $1\le i<j\le m$, with
  $u_{ij}\in\U(F)$. The conclusion of the Theorem then follows
  immediately.

  First, it is clear that $\gamma_n(F)$, $F''$, and
  $e_i[[X_i,X_j],u_{ij}]$ all belong to $F\cap(\varpi^n+\varpi\ss)$.

  Consider $v\in F\cap(\varpi^n+\varpi\rr)$. Then, as noted
  above, $v\in\gamma_2(F)$, so we may write
  \[v\equiv\sum_{1\le i<j\le m}[[X_i,X_j],u_{ij}]\pmod{F''}
  \]
  for some $u_{ij}\in\U(F)$. Furthermore, still working modulo $F''$,
  only the value in $\U(F/F')$ of the elements $u_{ij}$ is relevant;
  so that the variables $X_1,\dots,X_m$ may be arbitrarily permuted in
  their monomials. For $k<i$, we use the Jacobi identity
  $[[X_i,X_j],X_k]=[[X_k,X_j],X_i]-[[X_k,X_i],X_j]$ to rewrite the
  $u_{ij}$ in such a manner that no variable $X_k$ appears in
  $u_{ij}$; so $u_{ij}\in\U(\langle X_i,\dots,X_m\rangle)$.  We
  thus write
  \[v=v_1+\dots+v_{m-1},\text{ with }v_i\equiv\sum_{j>i}[[X_i,X_j],u_{ij}]\pmod{F''}.\]

  Consider then the endomorphism $\theta_i$ of $\U(F)$ defined by
  $\theta_i(X_j)=0$ for $j<i$, and $\theta_i(X_j)=X_j$ for $j\ge
  i$. The ideals $\ss$ and $\varpi$ are invariant under $\theta_i$, so
  applying $\theta_{m-1},\dots,\theta_1$ successively gives
  $v_i\in\varpi^n+\varpi\ss$ for all $i\in\{1,\dots,m-1\}$.  Next,
  modulo $\varpi^n+\varpi\ss$, we have
  \[0 \equiv v_i \equiv X_i\sum_{j>i} X_j u_{ij} - \sum_{j>i}X_jX_i u_{ij},
  \]
  and $\varpi$ is a free right $\U(F)$-module with basis
  $\{X_1,\dots,X_m\}$, so
  \begin{equation}\label{eq:deltan:3}
    X_i u_{ij}\in\varpi^{n-1}+\ss\text{ for all }i<j.
  \end{equation}
  Now $\ss=\sum_{i=1}^m e_iX_i\U(F)+\mathfrak a$, for the ideal
  $\mathfrak a=\U(F)\gamma_2(F)$, so
  \[\varpi^{n-1}+\ss=\sum_{i=1}^m X_i(\varpi^{n-2}+e_i\U(F))+\mathfrak a.\]
  Since $\U(F)/\mathfrak a=\Z[X_1,\dots,X_m]$ is an integral domain,
  \eqref{eq:deltan:3} yields
  $u_{ij}\in\varpi^{n-2}+e_i\U(F)+\mathfrak a$. Now $[[X_i,X_j],u]\in
  F''$ when $u\in\mathfrak a$, and $[[X_i,X_j],u]\in\gamma_n(F)$ when
  $u\in\varpi^{n-2}$; thus the proof is complete.
\end{proof}

Note, in particular, that if $e_i=0$ for all $i$ then we get the
following special case of the ``Fox problem'' (see~\S\ref{ss:fox}):
\begin{cor}
  For all $n\in\N$ we have
  \[F\cap(\varpi^n+\varpi\gamma_2(F))=\gamma_n(F)+F''.\]
\end{cor}

\subsection{\boldmath $n=4$} We continue with the notation set in the
paragraph preceding Theorem~\ref{thm:delta3} and proceed to give an
identification of the fourth Lie dimension subring.

\begin{theorem}\label{thm:delta4}
  The Lie subalgebra $(\varpi^4+\rr)\cap F$ of $F$ consists, modulo
  $\gamma_4(F)+R$, of all elements
  \begin{equation}\label{delta_4}
    \sum_{i>j}a_{ij}[X_i,X_j],
  \end{equation}
  with integer $a_{ij}$, such that $e_i$ divides $a_{ij}$ and for all
  $i\in\{1,\dots,m\}$
  \[W_i:=\sum_{i>j}a_{ij}X_j-\sum_{i<j}a_{ji}X_j\in e_i\gamma_2(F)+\gamma_3(F)+R.\]
\end{theorem}  
\begin{proof}
  Consider $w\in(\varpi^4+\rr)\cap F$. Then, for some $r\in R$, we
  have
  \[v:=w+r\in(\varpi^4+\varpi\rr)\cap F\subseteq 
  (\varpi^4+\varpi\ss)\cap F\subseteq\gamma_2(F).
  \]
  We may therefore write
  \begin{equation}\label{eq:delta_4:1}
    v\equiv\sum_{i>j} a_{ij}[X_i,X_j]+\sum_{i>j\le k}b_{ijk}[X_i,X_j,X_k]\mod \gamma_4(F);
  \end{equation}
  and, since $v\in\varpi^4+\varpi\ss$, we also have
  \[v\equiv\sum_{i,j}c_{ij}e_jX_iX_j+\sum_{i,j,k}d_{ijk}e_kX_iX_jX_k+\sum_{i\text{ any},j>k}f_{ijk}X_i[X_j,X_k]\mod \varpi^4.
  \]
  Comparing homogeneous terms of degree two, we have
  $a_{ij}=c_{ij}e_j=-c_{ji}e_i$. By Theorem~\ref{thm:deltan}, noting
  that $\gamma_4(F)$ contains $F''$, we may write
  $b_{ijk}=b'_{ijk}e_i$ with $b'_{ijk}\in\Z$; then
  \[b_{ijk}[X_i,X_j,X_k]=b_{ijk}'[e_iX_i,X_j,X_k]\in[R+\gamma_2(F),F,F]\subseteq\gamma_4(F)+R.
  \]
  Consequently,
  \[w\equiv\sum_{i>j}a_{ij}[X_i,X_j]\mod\gamma_4(F)+R.
  \]
  Define next $Y_i:=\sum_{j,k}d_{ijk}e_kX_jX_k+\sum_{j>k}f_{ijk}[X_j,X_k]$; then
  \[v\equiv\sum_iX_iW_i+\sum_iX_iY_i\mod\varpi^4,
  \]
  and $v\in\varpi^4+\varpi\rr$, so $W_i+Y_i\in\varpi^3+\rr$ for all
  $i\in\{1,\dots,m\}$. All degree-$3$ summands
  in~\eqref{eq:delta_4:1}, say involving the variables
  $\{X_i,X_j,X_k\}$, are multiples of $\gcd(e_i,e_j,e_k)$; so we may
  write $f_{ijk}=f_{ijk}^i e_i+f_{ijk}^j e_j+f_{ijk}^k e_k$ for some
  $f_{ijk}^i,f_{ijk}^j,f_{ijk}^k\in\Z$; then
  \begin{align*}
    Y_i &= \sum_{j>k}e_if^i_{ijk}[X_j,X_k]+\sum_{j,k}d_{ijk}X_j(e_kX_k)\\
    &\qquad+\sum_{j>k}f^j_{ijk}[e_jX_j,X_k]+\sum_{j>k}f^k_{ijk}[X_j,e_kX_k]\\
    &\in e_i\gamma_2(F)+\varpi^3+\rr;
  \end{align*}
  therefore $W_i\in e_i\gamma_2(F)+\varpi^3+\rr$. Noting then that
  $W_i$ belongs to $F$, we get
  $W_i\in(e_i\gamma_2(F)+\varpi^3+\rr)\cap
  F=e_i\gamma_2(F)+\gamma_3(F)+R$ by invoking Theorem~\ref{thm:delta3}.

  Conversely, choose any $a_{ij}\in\Z$ such that
  \[W_i:=\sum_{i>j}a_{ij}X_j-\sum_{i<j}a_{ji}X_j\in
  e_i\gamma_2(F)+\gamma_3(F)+R;\] then
  \begin{align*}
    \sum_{i>j}a_{ij}[X_i,X_j]&=\sum_i X_iW_i\\
    &\in\sum_i e_iX_i\gamma_2(F)+X_i\gamma_3(F)+X_iR \subseteq\varpi^4+\rr.\qedhere
  \end{align*}
\end{proof}

\begin{cor}
  If $L$ be a Lie algebra over $\Z$, then
  \[[\delta_4(L),L] \subseteq \gamma_5(L)+L'',\]
  with $L''$  the second derived subring of $L$. 

  In particular, if $L$ be a metabelian Lie ring, then
  $\delta_4(L)/\gamma_5(L)$ is central in $L/\gamma_5(L)$.
\end{cor}
\begin{proof}
  Consider a typical generator $[\sum_{i>j}a_{ij}[X_i,X_j],X_k]$ of
  $[\delta_4(L),L]$ modulo $\delta_5(L)$, resulting from the
  generators (\ref{delta_4}) of $\delta_4(L)$ as in Theorem
  \ref{thm:delta4}. Using the Jacobi identity, rewrite it as
  \begin{equation}
    \sum_{i>j}\Big(\frac{a_{ij}}{e_i}[e_iX_i,[X_j,X_k]]-\frac{a_{ij}}{e_j}[e_jX_j,[X_i,X_k]]\Big);
  \end{equation}
  since $e_iX_i,e_jX_j\in [L,L]$, the above element belongs to $L''$.
\end{proof}

\begin{cor}\label{2delta_4}
  If $L$ be a Lie algebra over $\Z$, then
  \[2\delta_4(L)\subseteq\gamma_4(L).\]
\end{cor}
\begin{proof}
  Consider a typical element $a=\sum_{i>j}a_{ij}[X_i,X_j]$ of
  $\delta_4(L)$ modulo $\gamma_4(L)$.  Recall our notation
  \[W_i:=\sum_{i>j}a_{ij}X_j-\sum_{i<j}a_{ji}X_j\in e_i\gamma_2(L)+\gamma_3(L),
  \]
  so that $a=\sum_iX_iW_i$.  Write $W_i=e_iY_i+Z_i$ with
  $Y_i\in\gamma_2(L)$ and $Z_i\in\gamma_3(L)$.  From
  $[X_i,X_j]=-[X_j,X_i]$, we also get $a=-\sum_iW_iX_i$. Therefore,
  \[2a=\sum_i[X_i,W_i]=\sum_i[e_iX_i,Y_i]+[X_i,Z_i]\in\gamma_4(L).\qedhere\]
\end{proof}

We briefly recall Magnus's construction alluded to in the
Introduction. Let $G$ be a group, and let $(\gamma_n(G))_{n\ge1}$ denote its
lower central series. The abelian group
\begin{equation}\label{eq:magnus}
  \gr G=\bigoplus_{n\ge1}\gamma_n(G)/\gamma_{n+1}(G)
\end{equation}
naturally has the structure of a graded Lie ring, for the Lie bracket
defined on homogeneous elements by
$[x\gamma_{m+1}(G),y\gamma_{n+1}(G)]=x^{-1}y^{-1}xy\gamma_{m+n+1}(G)$.

We show that, even though the fourth dimension quotient of $G$ may
non-trivial (namely $\delta_4(G)\neq\gamma_4(G)$), the corresponding
quotient of $\gr G$ is always trivial:

\begin{cor}\label{cor:graded=ok}
  If $L$ be a graded Lie ring, generated in degree $1$, then
  $\delta_4(L)=\gamma_4(L)$.
\end{cor}
\begin{proof} 
  Since $\delta_3(L)=\gamma_3(L)$ by Theorem~\ref{thm:delta3}, the
  generators of $\delta_4(L)/\gamma_4(L)$ have degree at least $3$;
  however, by Theorem~\ref{thm:delta4}, they have degree
  $2$. Therefore, $\delta_4(L)/\gamma_4(L)=0$.
\end{proof}

\subsection{A counterexample in degree \boldmath $4$}
We now give an example of a Lie ring for which
$\delta_4(L)\neq\gamma_4(L)$. This is an adaptation of Rips's
counterexample~\cite{Rips:72} to the dimension conjecture. Note
however that, by Corollary~\ref{cor:graded=ok}, the Magnus Lie algebra
$\gr G$ associated with Rips's counterexample \emph{does} satisfy
$\delta_4(\gr G)=\gamma_4(\gr G)$.

\begin{theorem}\label{case n=4}
  Consider the  Lie ring $L$ with presentation
  \[\begin{split}
    \langle x_1,x_2,x_3,x_4\mid & 4x_1 + 2[x_4,x_3] + [x_4,x_2],\\
    &  16x_2 + 4[x_4,x_3] - [x_4,x_1],\, 64x_3 - 4[x_4,x_2] - 2[x_4,x_1].\rangle
  \end{split}
  \]
  Set $a:=32[x_1,x_2] + 64 [x_1,x_3] + 128[x_2,x_3]\in L$. We then
  have
  \[a\in\delta_4(L)\setminus\gamma_4(L).\]
\end{theorem}
\begin{proof}
  First, observe that $a\in\delta_4(L)$ by Theorem~\ref{thm:delta4}.

  We may then seek a quotient $\overline L$ of $L$, nilpotent of class
  $3$, in which $a$ does not vanish. We add as relations to $L$:
  \begin{itemize}
  \item all triple commutators except $[x_4,x_i,x_i]$ for $i=1,2,3$;
  \item $4[x_4,x_1,x_1]$, $4[x_4,x_2,x_2]-[x_4,x_1,x_1]$ and
    $4[x_4,x_3,x_3]-[x_4,x_2,x_2]$;
  \item all quadruple commutators (namely, $\gamma_4(L)$).
  \end{itemize}

  It is then easily checked that $\overline L$ is, additively, a
  $\Z$-module of rank $9$, of the form
  \[\Z\oplus\Z/256\oplus\Z/256\oplus\Z/256\oplus\Z/16\oplus\Z/16\oplus\Z/8\oplus\Z/4\oplus\Z/2,\]
  and that $a$ is a non-trivial element in $\overline L$.  These
  calculations on finite-rank Lie algebras were performed using
  GAP~\cite{gap4:manual} and its package
  \textsc{LieRing}~\cite{liering:manual} by Willem de Graaf and Serena
  Cical\`o.
\end{proof}

\section{The Fox Problem}\label{ss:fox}
Analogous to the Fox problem for free group
rings~\cite{Gupta:87}*{Chapter~III}, there arises the following
\begin{problem}\label{fox}
  Identify the Lie subrings $F\cap\varpi^n(F)\rr$, for $n\geq
  1$, where $F$ is a free Lie ring over a commutative ring $\K$ with
  identity and $\rr$ is the two-sided ideal of $\U(F)$
  generated by a Lie ideal $R$ of $F$.
\end{problem}
Note that the subrings $F\cap\varpi^n(F)\rr$ and $F\cap\varpi^n\rr$
are anti-isomorphic; indeed the map $x\mapsto-x$ on $F$ extends to an
anti-isomorphism on $\U(F)$, the \emph{antipode} (see
e.g.~\cite{Jacobson:1962}*{Theorem~V.1}).

In case $\K$ is a field, a solution to the above problem is provided by
the following
\begin{theorem}[Yunus~\cite{Yunus:84}]\label{thm:yunus}
  Let $F$ be a free Lie algebra over a field $\K$, let $R$ be an ideal
  of $F$, and let $\rr$ be the ideal of $\U(F)$ generated by
  $R$. Then, for all $n\ge1$,
  \[F\cap \varpi^n(F)\rr=\sum_{\substack{2\leq m\leq n+1\\
      \sum_j i_j\geq n+\max_j i_j}} [R_{i_1},\ldots,R_{i_m}],
  \]
  with $R_i=R\cap \gamma_i(F)$.
\end{theorem}

The problem has also been solved in the case of restricted Lie
algebras~\cites{Usefi:2008,Usefi:2008a}; however, the integral case
($\K=\Z$) of Problem~\ref{fox} does not seem to have been investigated
so far. The very first case here ($n=1$) manifests a sharp difference
with the corresponding result for group rings. It is well-known that
if $N$ is a normal subgroup of a group $G$, then
\[G\cap (1+\varpi(G)\varpi(N))=[N,N].
\]
In contrast with this result, we note the following:

\begin{example}\label{example:fox}
  Let $F$ be the free $\Z$-free Lie algebra with basis $\{X_1,X_2\}$,
  let $R=pF$ be the Lie ideal of $F$ generated by $\{pX_1,pX_2\}$, and
  let $\rr$ be the two-sided ideal of $\U(F)$ generated by $R$. Then
  \[F\cap\varpi(F)\rr=p[F,F]\neq [R,R]=p^2[F,F].
  \]
  However, with $\sqrt{[R,R]}=\{x\in R\colon nx\in[R,R]\text{ for some
  }n\neq0\}$, we always have
  \[[R,R]\subseteq F\cap\varpi(F)\rr\subseteq\sqrt{[R,R]}\subseteq R.\]
\end{example}
\begin{proof}
  The element $p[X_1,X_2]=X_1(pX_2)-X_2(pX_1)$ belongs to $\varpi\rr$,
  so the ideal $p[F,F]$ that it generates is contained in
  $F\cap\varpi\rr$. Conversely, $F\cap\varpi\rr$ is contained in
  $[F,F]$ and in $pF$, so is contained in their intersection
  $p[F,F]$. Finally, we clearly have $[R,R]=p^2[F,F]$.

  The second claim follows from Theorem~\ref{thm:yunus}.
\end{proof}

Let us now derive some consequences of the identification of
$L\cap\varpi(L)\rr$. Consider a free presentation $L=F/R$, the
associated ideal $\rr=\U(F)R$, and write $\varpi:=\varpi(F)$. To avoid
trivial exceptional cases, we assume that $F$ is non-abelian, or
equivalently that its rank is at least $2$, and that $R$ is
non-zero. Set
\[M:=F\cap\varpi\rr,\qquad \widetilde L=F/M.
\]
The natural Lie monomorphism $\iota:F\to\U(F)$ induces a Lie
monomorphism
\[\iota\colon F/M\to\varpi/\varpi\rr,
\]
which restricts to a monomorphism of right $\U(L)$-modules
$\iota\colon R/M\to\varpi/\varpi\rr$. Observe that $\varpi/\varpi\rr$
is a free right $\U(L)$-module with basis $\{X_i+\varpi\rr\}_{1\le i\le
  m}$. We write as before $[-,-]$ for the adjoint action of $\U(L)$.

By Theorem~\ref{thm:yunus}, we have $M=[R,R]$ if $L$ is an algebra
over a field. In analogy with the case of groups, we call $R/M$ the
\emph{relation module} of the presentation $L=F/R$.
\begin{prop}\label{prop:faithful}
  Let $L$ be a Lie algebra. If $\varpi$ is an integral domain (for
  example, if $L$ is an algebra over a field), then $R/M$ is a
  faithful $\U(L)$-module.
\end{prop}
\begin{proof}
  We have, by definition of $M$, an embedding
  $R/M\hookrightarrow\rr/\varpi\rr$; moreover, there is a natural
  embedding $\rr/\varpi\rr\hookrightarrow\varpi/\varpi\rr$, and there
  is an isomorphism
  \[\left\{\begin{array}{rl}
    \varpi/\varpi\rr &\cong\varpi/\varpi^2\otimes\U(L),\\
    \displaystyle\sum_{i=1}^m X_iu_i+\varpi\rr &\mapsto (X_i+\varpi^2)\otimes(u_i+\rr)\text{ with }u_i\in\U(F).\end{array}\right.
  \]
  Composing these maps, we see that $R/M$ embeds in the free
  $\U(L)$-module $\varpi/\varpi^2\otimes\U(L)$. Now the assumption
  that $\U(L)$ is an integral domain implies that all non-zero
  submodules of a free $\U(L)$-module are faithful.

  Recall that $\U(L)$ is a domain if its associated graded
  $\operatorname{Sym}(L)$ is a domain; this holds e.g.\ when $L$ is an
  algebra over a field.
\end{proof}

It may be noted, continuing on Example~\ref{example:fox}, that in that
case $M=p[F,F]$ so $R/M\cong\Z^2$ with basis $\{pX_1,pX_2\}$; so that
$R/M$ is annihilated by $\varpi(L)$. Thus, in contrast with the case
for groups, $R/M$ is, in general, not a faithful $\U(L)$-module.

Note also, when $\Bbbk=\Z$, that $\varpi$ is an integral domain if and
only if $L$ either is torsion-free or satisfies $pL=0$ for some prime
$p$.

Note finally, as in the classical Gasch\"utz theory for groups, we
have an exact sequence
\[\begin{tikzcd}0\arrow{r} & R/M\arrow{r} & \varpi/\varpi^2\otimes\U(L)\arrow{r} & \varpi(L)\arrow{r} & 0\end{tikzcd}.
\]

\newcommand\ann{\operatorname{ann}}

\begin{theorem}
  Let as above $L$ be a Lie algebra presented as $F/R$ with $R\neq0$
  and $F$ non-abelian. Set $M=F\cap\varpi R$ and $\widetilde
  L=F/M$.
  \begin{enumerate}
  \item[(i)] If $\bigcap_{n\ge1}\varpi^n(L)=0$, then
    $\bigcap_{n\ge1}\gamma_n(\widetilde L)=0$.
  \item[(ii)] If $\bigcap_{n\ge1}\gamma_n(\widetilde L)=0$ and the
    annihilator $\ann_{\U(L)}(R/M)$ is trivial, then
    $\bigcap_{n\ge1}\varpi^n(L)=0$.
  \item[(iii)] In particular, if $\U(L)$ is an integral domain (for
    example, if $L$ is an algebra over a field), then
    $\bigcap_{n\ge1}\varpi^n(L)=0$ if and only if
    $\bigcap_{n\ge1}\gamma_n(\widetilde L)=0$.
  \end{enumerate}
\end{theorem}
\begin{proof}
  Suppose first $\bigcap_{n\ge1}\varpi^n(L)=0$, and consider
  $w\in\bigcap_{n\ge1}\gamma_n(\widetilde L)$. Clearly
  $\bigcap_{n\ge1}\gamma_n(L)=0$, because
  $\gamma_n(L)\subseteq\varpi^n(L)$, so $w\in R/M$. Writing
  $\rr=\U(F)R$, note that
  \[\iota(w)\in\bigcap_{n\ge1}(\varpi^n+\varpi\rr)/\varpi\rr=\sum_{i=1}^m\big[X_i+\varpi\rr,\bigcap_{n\ge1}\varpi^n(L)\big]=0.\]
  Since $\iota$ is injective, we deduce $w=0$, which proves~(i).

  Suppose next $\bigcap_{n\ge1}\gamma_n(\widetilde L)=0$, and consider
  $z\in\bigcap_{n\ge1}\varpi^n(L)$. We then have
  $[R/M,z]\subseteq\bigcap_{n\ge1}\gamma_n(\widetilde L)=0$, so $z$ belongs
  to the annihilator $\ann(R/M)$. This proves~(ii).

  (iii) follows immediately from Proposition~\ref{prop:faithful}.
\end{proof}

We remark that, in the case of groups $G=F/R$, we always have
$M=[R,R]$, and $R/M$ is a faithful $\Z[G]$-module,
see~\cite{Passi:75}.

\section{A simplicial approach}
In this section, we formulate an approach to the investigation of the
dimension problem using simplicial methods. We adapt to Lie rings, on
the one hand, the spectral-sequence methods developed by Gr\"unenfelder
for group rings and, on the other hand, Keune's theory of derived
functors of the functors $\varGamma_n$.

As a byproduct, we will reprove Theorem~\ref{thm:delta3}. We hope,
however, that these methods will bear more fruits in later work. We
follow the notation in~\cite{MP:09}*{Appendix~A}; see
also~\cite{Ellis:93}. General valuable references for simplicial
methods include~\cite{May:67} and~\cite{Goerss-Jardine:99}.

Fortunately, Lie rings are abelian groups equipped with the extra
structure of a Lie bracket; therefore, all the machinery developed for
groups applies readily. We have attempted to give sufficient detail so
as to make the text self-sufficient, and refer the reader to the cited
literature for details if needed.

\newcommand\Lie{\mathfrak{Lie}}
\newcommand\sLie{\mathfrak{sLie}}
\newcommand\asLie{\mathfrak{asLie}}
\newcommand\Ass{\mathfrak{Ass}}
\newcommand\sAss{\mathfrak{sAss}}

Let $\Lie$ denote the category of \emph{Lie rings $L$ over $\Z$}, and
let $\sLie$ denote the category of \emph{simplicial Lie rings $\mathbf
  L=\{L_n\}_{n\geq 0}$ over $\Z$}. We recall briefly its definition:
first, the \emph{simplicial category} $\Delta$ is the category with
object set $\N$, and with $\hom_\Delta(m,n)$ equal to the set of
order-preserving maps from $\{0,\dots,m-1\}$ to $\{0,\dots,n-1\}$. A
simplicial Lie ring $\mathbf L$ is a contravariant functor
$\Delta\to\Lie$; equivalently, for every $n\in\N$, a Lie ring
$L_n=\mathbf L(n)$, and \emph{face} maps $d_0,\dots,d_n:L_n\to
L_{n-1}$ and \emph{degeneracy} maps $s_0,\dots,s_n:L_n\to L_{n+1}$
between the Lie rings, satisfying appropriate conditions (see
e.g.~\cite{May:67}*{Definition~1.1}).

The associated \emph{Moore complex} $\mathbf M$ is the complex of Lie
rings $M_n=\bigcap_{i=0}^{n-1}\ker(d_i|L_n)$; its differential is
given by $d_n:M_n\to M_{n-1}$. The \emph{fundamental Lie rings}
$\pi_n(\mathbf L)$ are the homology groups of the Moore complex:
$\pi_n(\mathbf L)=\ker(d_n|M_n)/\operatorname{im}(d_{n+1}|M_{n+1})$.

Since, in particular, Lie rings are abelian groups, we may define
alternatively the fundamental Lie rings as follows. Form the
total differential
\[\partial_n:L_n\to L_{n-1},\qquad\partial_n=\sum_{i=0}^n(-1)^i d_i.\]
Then $\partial^2=0$, and $\pi_n(\mathbf L)$ is isomorphic to the
homology of the resulting complex $(\mathbf L,\partial)$, via the
natural inclusions $M_n\to L_n$.

A \emph{fibration} is a morphism of simplicial Lie rings $\mathbf
L\to\mathbf Q$ that is surjective in all degrees. Its \emph{fibre} is
the simplicial Lie ring $\mathbf K$ whose Lie rings in respective
degrees are the kernels of the aforesaid morphisms. As is well known,
a fibration gives rise to a \emph{long exact homotopy sequence} of
simplicial Lie rings, namely the long exact sequence
\[\cdots\to \pi_1(\mathbf K)\to \pi_1(\mathbf L)\to \pi_1(\mathbf Q)\to
\pi_0(\mathbf K)\to \pi_0(\mathbf L)\to \pi_0(\mathbf Q)\to 0.
\]

The zeroth homotopy ring is usually readily computable, in contrast
with higher homotopy rings $\pi_n(\mathbf L)$ for $n\geq 1$. Again, this
follows from the classical result for abelian groups:
\begin{prop}\label{prop:coequalizer}
  $\pi_0(\mathbf L)$ is the coequalizer of
  $d_0,d_1:L_1\rightrightarrows L_0$.\qed
\end{prop}

The Lie algebra $L_0$ acts by derivations on $L_n$, via
\[[y,x]=[s_0^n(y),x]\text{ for all }y\in L_0,x\in L_n.
\]
This action preserves the ideal $M_n$, and induces an action on
$\pi_n(\mathbf L)$. Thus
\begin{prop}
  The Lie ring $\pi_n(\mathbf L)$ has the natural structure of a
  $\U(\pi_0(\mathbf L))$-module.\qed
\end{prop}

Assume that $L_2$ is generated by degeneracies: $L_2=\langle
s_0(L_1),s_1(L_1)\rangle$. Then $d_2(M_2)= [\ker(d_1),\ker(d_2)]$, whence
\begin{prop}
  If $L_2$ is generated by degeneracies, then
  \[\pi_1(\mathbf L)=\frac{\ker(d_0|L_1)\cap\ker(d_1|L_1)}{[\ker d_0,\ker d_1]}.\qed\]
\end{prop}

\subsection{Two spectral sequences}\label{ss:ss}
We take inspiration from Gr\"unenfelder's
approach~\cite{Gruenenfelder:80}, via spectral sequences, to the
dimension subgroup problem.  For every integer $n\geq 1$, we have
functors
\[\varGamma_n:\Lie\to \Lie,\quad L\mapsto L/\gamma_n(L)
\]
and
\[\gr_n:\Lie\to\Lie, \quad L\mapsto \gamma_n(L)/\gamma_{n+1}(L).
\]
These functors extend naturally to the category $\sLie$, and we have,
for $\mathbf L\in \sLie$ and all $n\geq 1$, an exact sequence of
simplicial Lie rings
\begin{equation}\label{eq:gr_m sequence}
  \begin{tikzcd}0\arrow{r} & \gr_n(\mathbf L)\arrow{r} &
  \varGamma_{n+1}(\mathbf L)\arrow{r} & \varGamma_n(\mathbf L)\arrow{r} & 0.
\end{tikzcd}
\end{equation}
The resulting homotopy exact couple
\[\begin{tikzcd}[column sep=0mm]
  \pi_*(\varGamma_*(\mathbf L))\arrow{rr} &&
  \pi_*(\varGamma_*(\mathbf L))\arrow{dl}\\
  & \pi_*(\gr_*(\mathbf L))\arrow{ul}
\end{tikzcd}
\]
yields a first-quadrant spectral sequence $\{E^r_{p,q}(\mathbf L)\}$
having differentials $d^r$ of bidegree $(r,-1)$, and with
\begin{equation}\label{eq:definition E^1}
  E^1_{p,q}(\mathbf L)=\pi_q(\gr_p(\mathbf L)).
\end{equation}
This is the classical sequence associated with the filtration of $L$
by its lower central series $\{\gamma_n(L)\}$; thus
\begin{prop}[Essentially~\cite{Curtis:71} and~\cite{Gruenenfelder:80}]\label{prop:E^infty}
  Let $L=F/R$ be a nilpotent Lie algebra, and let $E^*_{*,*}$ be the
  above spectral sequence. Then $E^*_{*,*}\Rightarrow\gr\pi_*(\mathbf
  L)$. In particular,
  \[E^\infty_{n,0}=\frac{\gamma_n(L)}{\gamma_{n+1}(L)}=\frac{\gamma_n(F)+R}{\gamma_{n+1}(F)+R}\text{ for all }n\ge1.\qed\]
\end{prop}

We consider now a parallel construction in the context of the
universal enveloping algebra of $L$.  Let $\Ass$ denote the category
of augmented associative algebras over $\Z$, and let $\sAss$ denote
the category of simplicial augmented algebras. If $\mathbf L$ be a
simplicial Lie ring, then $\U(\mathbf L)$ is the simplicial
associative algebra with $\U(L_n)$ in degree $n$. As above, for every
integer $n\geq 1$ we have functors
\[\overline{\varGamma}_n:\Lie\to \Ass,\quad L\mapsto \U(L)/\varpi^n(L)
\]
and
\[\overline\gr_n:\Lie\to\Ass, \quad L\mapsto \varpi^n(L)/\varpi^{n+1}(L),\]
where, as usual, $\varpi(L) $ denotes the augmentation ideal of the
augmented algebra $\U(L)$. These functors also extend naturally to the
category $\sAss$ of simplicial augmented algebras, and we have for all
$n\geq 1$ an exact sequence of simplicial algebras
\begin{equation}\label{eq:bar gr_m sequence}
  \begin{tikzcd}
    0\arrow{r} & \overline\gr_n(\mathbf L)\arrow{r} &
    \overline\varGamma_{n+1}(\mathbf L)\arrow{r} &
    \overline\varGamma_n(\mathbf L)\arrow{r} & 0.
  \end{tikzcd}
\end{equation}
The resulting homotopy exact couple
\[\begin{tikzcd}[column sep=0mm]
  \pi_*(\overline\varGamma_*(\mathbf L))\arrow{rr} &&
  \pi_*(\overline\varGamma_*(\mathbf L))\arrow{dl}\\
  & \pi_*(\overline\gr_*(\mathbf L))\arrow{ul}
\end{tikzcd}
\]
yields another first-quadrant spectral sequence
$\{\overline{E}^r_{p,q}(\mathbf L)\}$ having differentials
$\overline{d}^r$ of bidegree $(r,-1)$, with
\begin{equation}\label{eq:definition bar E^1}
  \overline{E}^1_{p,q}(\mathbf L)=\pi_q(\overline\gr_p(\mathbf L)).
\end{equation}

\begin{prop}[Essentially~\cite{Curtis:71} and~\cite{Gruenenfelder:80}]\label{prop:bar E^infty}
  Let $L=F/R$ be a nilpotent Lie algebra, and let $\overline
  E^*_{*,*}$ be the above spectral sequence. Then $\overline
  E^*_{*,*}\Rightarrow\gr\pi_*(\U\mathbf L)$. In particular,
  \[\overline E^\infty_{n,0}=\frac{\varpi^n(L)}{\varpi^{n+1}(L)}=\frac{\varpi^n(F)+\rr}{\varpi^{n+1}(F)+\rr}\text{ for all }n\ge1.\qed\]
\end{prop}

In view of Theorem~\ref{thm:PBW}, for every $\mathbf L\in \sLie$ there
exists a canonical embedding $\boldsymbol\iota:\mathbf L\to \U(\mathbf
L)$. We thus obtain a morphism of spectral sequences
\[\boldsymbol\iota\colon E^*_{*,*}\to\overline E^*_{*,*}.\]

\subsection{Resolutions}
Recall that an \emph{augmented simplicial object} in the category
$\sLie$ is a simplicial Lie ring $\mathbf L$ together with a Lie ring
$L_{-1}\in \Lie$ and a Lie homomorphism $d_0:L_0\to L_{-1}$, such that
$d_0d_0=d_0d_1:L_1\to L_{-1}$.  Let $\asLie$ denote the category of
\emph{augmented simplicial Lie rings}.  For $\mathbf L\in\asLie$ and
$n\geq -1$, set
\[Z_n\mathbf L=\{(a_0,\ldots,a_{n+1})\in
L_n^{n+2}\mid d_ia_j=d_{j-1}a_i\text{ for all }i<j\}.
\]
For each $n\geq-1$, there is a homomorphism
\[d^{(n)}:L_{n+1}\to Z_n\mathbf L,\qquad a\mapsto (d_0a,\ldots,d_{n+1}a).
\]
An augmented simplicial object $\mathbf L\in \asLie$ is called
\emph{aspherical} if, for each $n\geq -1$, the homomorphism $d^{(n)}$
is surjective. In this case $\mathbf L$ is said to be a
\emph{resolution} of $L_{-1}$.  An object $\mathbf L\in \asLie$ is
said to be \emph{free} if, for each $n\geq 0$, the Lie ring $L_n\in
\Lie$ is free and if, moreover, sets $X_n\subset L_n$ of free
generators can be chosen in such a way that they are preserved under
the degeneracy maps $s_i$, namely, $s_iX_n\subset X_{n+1}$ for all
$0\leq i\leq n$.

The following statement follows directly from the fact that $\sLie$ is
a closed model category in the sense of Quillen~\cite{Quillen:67}; we
include its proof because it is informative.
\begin{prop}[F. Keune,~\cite{Keune:73}]\label{prop:free resolution}
  Every Lie ring $L\in \Lie$ has a free resolution $\mathbf L\in \asLie$.
\end{prop}
\begin{proof}
  Let $L$ be a Lie ring. Choose a set of generators $X_0$ of $L$, and
  define $L_0=FX_0$, the free Lie ring over $X_0$. The surjection
  $d_0:FX_0\to L$ is then defined by $d_0(x)=x$ for all $x\in
  X_0$. Suppose that $L_n=FX_n$ has been defined, and that we wish to
  construct $L_{n+1}=FX_{n+1}$ with homomorphisms $d_i:L_{n+1}\to L_n$
  and $s_j:L_n\to L_{n+1}$.  Set
  \[Z=\{(a_0,\ldots,a_{n+1})\in L_n^{n+2}\mid d_ia_j=d_{j-1}a_i\text{ for all } i<j\}.
  \]
  There are $n+2$ homomorphisms $d'_i:Z\to L_n$, defined by
  $d'_i(a_0,\ldots,a_{n+1})=a_i$ for all $i=0,\ldots,n+1,$ and $n+1$
  homomorphisms $s'_i:L_n\to Z$, defined
  by
  \[s'_i(a)=(s_{i-1}d_0a,\ldots,s_{i-1}d_{i-1}a,a,a,s_id_{i+1}a,\ldots,s_id_na)\]
  for $i=0,\ldots,n$. Complete the subset
  $\bigcup_{i=0}^ns'_iX_n\subset Z$ to a set $X_{n+1}$ of generators
  of $Z$. Define
  \[L_{n+1}=FX_{n+1},\quad d_i\text{ extending }d'_i:L_{n+1}\to L_n,\quad s_i\text{ extending }s'_i:L_n\to L_{n+1}.
  \]
  Thus we obtain an augmented simplicial object $\mathbf L$ which by
  its construction is a free resolution of $L\in \Lie$.
\end{proof}

\begin{note}\label{note:free resolution}
  We extract from the proof above a concrete construction of a free
  resolution $\mathbf L$ for any Lie ring given by a presentation. Let
  indeed $L=\langle X\mid\mathscr R\rangle$ be a presentation, meaning
  $L=F(X)/R$ with $R=\langle\mathscr R\rangle$, the ideal in $F(X)$
  generated by $\mathscr R$.

  A free resolution $\mathbf L$ of $L$ may then start as follows. Set
  $L_0=F(X)$ and $L_1=F(X\sqcup\mathscr R)$. The map $L_0\to L$ is the
  natural quotient map. The differentials $d_0,d_1:L_1\to L_0$ are
  defined on generators by $d_0(x)=d_1(x)=x$ for $x\in X$ and
  $d_0(r)=0,d_1(r)=r$ for $r\in\mathscr R$, while the degeneracy
  $s_0:L_0\to L_1$ is the natural embedding $s_0(x)=x$.
\end{note}

\subsection{Kan's condition}
We now show that objects in $\sLie$ satisfy Kan's \emph{extension
  condition}. Let $\mathbf L\in \sLie$ be a simplicial Lie ring, and
let $y_0,\ldots,\hat{y}_k ,\ldots,y_n$ be a collection of simplices in
$L_{n-1}$ satisfying $d_iy_j=d_{j-1}y_i $ for $k\neq i<j\neq
k$. Define
\[\begin{cases}
  w_0=s_0y_0,\\
  w_i=w_{i-1}-s_id_iw_{i-1}+s_iy_i & \text{ for }0<i<k,\\
  w_n=w_{k-1}-s_{n-1}d_nw_{k-1}+s_{n-1}y_n,\\
  w_i=w_{i+1}-s_{i-1}d_iw_{i+1}+s_iy_i &\text{ for }k<i<n.
\end{cases}
\]
Then $w_{k+1}\in L_n$ and $d_iw_{k+1}=y_i$ for $i\neq k$.  Thus
$\mathbf L$ satisfies the Kan extension
condition~\cite{Curtis:71}*{Proposition~1.13 and~Lemma 3.1}.  As a
consequence, the following comparison theorem holds.
\begin{theorem}[F. Keune,~\cite{Keune:73}]\label{thm:comparison}
  Let $\mathbf L\in\asLie$ be free, let $\mathbf M\in\asLie$ be
  aspherical, and let $\alpha:L_{-1}\to M_{-1}$ be a Lie
  homomorphism. Then
  \begin{enumerate}
  \item there exists a homomorphism $\boldsymbol\alpha:\mathbf L\to
    \mathbf M$ in $\asLie$ such that $\alpha_{-1}=\alpha$;
  \item if $\boldsymbol\alpha'$ also extends $\alpha$, then there
    exists a simplicial homotopy
    $h:\boldsymbol\alpha\simeq\boldsymbol\alpha'$, with the $h_i$ Lie
    homomorphisms in $\Lie$.
  \end{enumerate}
\end{theorem}

\subsection{Derived functors}
Let $\mathcal F:\Lie\to \Lie$ be a right-exact functor. Its left
derived functors are defined as follows, see~\cite{TV:1969}.  In view
of Theorem~\ref{thm:comparison}, the homotopy groups $\pi_n(\mathcal
F\mathbf L)$ for a free resolution $\mathbf L\to L$ depend only on
$L\in \Lie$ and not on the chosen free resolution $\mathbf L$ of
$L$. The left derived functors $\mathfrak D_n \mathcal F:\Lie\to
\Lie$, for $n\geq 0$, are defined by setting
\[\mathfrak D_n\mathcal F(\mathbf L)=\pi_n(\mathcal F\mathbf L).\]

Let $R$ be an ideal of a free Lie ring $F$. We introduce the following
notation, for all $n\geq0$:
\begin{xalignat*}{2}
  R(0)&=R, & R(n+1)&=[R(n),F]\\
  \rr(0)&=\U(F)R, & \rr(n+1)&=\varpi(F)\rr(n)+\mathfrak
r(n)\varpi(F).
\end{xalignat*}

Recall that, for $n\geq 1$, the functor $\varGamma_n:\Lie\to\Lie$ is
given by $L\mapsto L/\gamma_n(L)$.  Let $L=F/R$ be a free presentation
of $L$, and let $\mathbf L$ be a free simplicial resolution of $L$
with the first two terms $F_0$ and $F_1$ as in Note~\ref{note:free
  resolution}.  Then clearly $\mathfrak
D_0\varGamma_n(L)=L/\gamma_n(L)$ for all $n\geq 2$. Consider the short
exact sequence
\[\begin{tikzcd}
  0\arrow{r} & \gamma_n(\mathbf L)\arrow{r} & \mathbf L\arrow{r} & \mathbf L/\gamma_n(\mathbf L)\arrow{r} & 0
\end{tikzcd}\]
of simplicial Lie rings.  It gives rise to a long exact homotopy
sequence, namely
\[\cdots\to \pi_1(\mathbf L)\to \pi_1(\mathbf L/\gamma_n(\mathbf L))\to \pi_0(\gamma_n(\mathbf L))\to \pi_0(\mathbf L)\to 
\pi_0(\mathbf L/\gamma_n(\mathbf L))\to 0.
\]
Now $\pi_1(\mathbf L)=0$, and the $\pi_0$ may be computed via
Proposition~\ref{prop:coequalizer} as coequalizers:
\begin{align*}
  \pi_0(\mathbf L) &= F/R=L,\\
  \pi_0(\mathbf L/\gamma_n(\mathbf L)) &= F/(R+\gamma_n(F))=L/\gamma_n(L),\\
  \pi_0(\gamma_n(\mathbf L)) &= \gamma_n(F)/\gamma_n(F*R)=\gamma_n(F)/R(n-1).
\end{align*}
Hence
\begin{equation}\label{eq:D_1}
  \mathfrak D_1\varGamma_n(L)=\pi_1(\mathbf L/\gamma_n(\mathbf L))=\ker\big[\pi_0(\gamma_n(\mathbf L))
  \to \pi_0(\mathbf L)\big]=\frac{R\cap\gamma_n(F)}{R(n-1)}.
\end{equation}
We deduce in passing that $R\cap\gamma_n(F)/R(n-1)$ depends only on
$L$, and not on the choice of presentation $F/R$; they are known, in
the context of groups, as \emph{Baer invariants}. See~\cite{CB:1990}
for more details.

The sequence~\eqref{eq:gr_m sequence} yields a long exact homotopy sequence
\begin{alignat*}{2}
  \cdots&\to\pi_1(\mathbf L/\gamma_{n+1}(\mathbf L))\to
  \pi_1(\mathbf L/\gamma_n(\mathbf L))\to\\
  \to\pi_0(\gamma_n(\mathbf L)/\gamma_{n+1}(\mathbf L))
  &\to\pi_0(\mathbf L/\gamma_{n+1} (\mathbf L))\to
  \pi_0(\mathbf L/\gamma_n(\mathbf L))\to 0;
\end{alignat*}
all the terms above are known, except the middle one which we now
proceed to determine. Here is the same sequence, presented as a splice
of (diagonal) short exact sequences, in which we abbreviate $\gamma_m$
for $\gamma_m(F)$:
\[\begin{tikzcd}[bo column sep=4em]
  &&& \frac{R\cap\gamma_n}{R(n-1)+(R\cap\gamma_{n+1})}\arrow{dr} &&&& \frac F{R+\gamma_n}\\
  \frac{R\cap\gamma_{n+1}}{R(n)}\arrow{rr}\arrow{dr} && \frac{R\cap\gamma_n}{R(n-1)}\arrow{rr}\arrow{ur} && \pi_0(\gr_n L)\arrow{rr}\arrow{dr} && \frac F{R+\gamma_{n+1}}\arrow{ur}\\
  & \frac{R\cap\gamma_{n+1}}{R(n-1)\cap\gamma_{n+1}}\arrow{ur} &&&& \frac{R+\gamma_n}{R+\gamma_{n+1}}\arrow{ur}\hbox to 0pt{ $=\frac{\gamma_n}{\gamma_n\cap(R+\gamma_{n+1})}$}
\end{tikzcd}\]
whence
\begin{equation}\label{eq:E^1}
  E^1_{n,0}=\mathfrak D_0\gr_m(L)=\frac{\gamma_n(F)}{\gamma_{n+1}(F)+R(n-1)}\text{ for all }n\geq 2.
\end{equation}

\subsection{Functors to associative algebras}
Consider next, for $n\geq 1$, the functor $\overline\varGamma_n:\Lie\to
\Ass$ given by $L\mapsto \U(L)/\varpi^n(L)$. Let $L=F/R$ be a free
presentation of $L$ and let $\mathbf L$ a free simplicial resolution
of $L$ with the first two terms $F_0$ and $F_1$ as in
Note~\ref{prop:free resolution}.  Then clearly
\[\mathfrak D_0\overline\varGamma_n(L)=\U(L)/\varpi^n(L)\text{ for all }n\geq 2.\]
Consider the short exact sequence
\[\begin{tikzcd}
  0\arrow{r} & \varpi^n(\mathbf L)\arrow{r} & \U(\mathbf L)\arrow{r} & \U(\mathbf L) /\varpi^n(\mathbf L)\arrow{r} & 0
\end{tikzcd}\]
of simplicial algebras.  It gives rise to a long exact homotopy
sequence, namely
\begin{alignat*}{2}
  \cdots &\to \pi_1(\U(\mathbf L))\to \pi_1(\U(\mathbf
  L)/\varpi^n(\mathbf L))\to\\
  \to\pi_0(\varpi^n(\mathbf L)) &\to \pi_0(\U(\mathbf L))\to
  \pi_0(\U(\mathbf L)/\varpi^n(\mathbf L))\to 0.
\end{alignat*}
Now $\pi_1(\U(\mathbf L))=0$, and the $\pi_0$ can be computed via
Proposition~\ref{prop:coequalizer} as coequalizers:
\begin{align*}
  \pi_0(\U(\mathbf L)) &= \U(F)/\rr = \U(L),\\
  \pi_0(\U(\mathbf L)/\varpi^n(\mathbf L)) &= \U(F)/(\varpi^n(F)+\rr) = \U(L)/\varpi^n(L),\\
  \pi_0(\varpi^n(\mathbf L)) &= \varpi^n(F)/\rr(n-1).\\
\end{align*}
Hence
\begin{equation}\label{eq:bar D_1}
  \mathfrak D_1\overline\varGamma_n(L)=\pi_1(\U(\mathbf L)/\varpi^n(\mathbf L))=\ker\big[\pi_0(\varpi^n(\mathbf L))\to\pi_0(\U(\mathbf L))\big]=\frac{\rr\cap \varpi^n(F)}{\rr(n-1)}.
\end{equation}
The sequence~\eqref{eq:bar gr_m sequence} yields a long exact homotopy
sequence
\begin{alignat*}{2}
  \cdots &\to \pi_1(\U(\mathbf L)/\varpi^{n+1}(\mathbf L))\to \pi_1(\U(\mathbf L)/\varpi^n(\mathbf L))\to\\
  \to \pi_0(\varpi^n(\mathbf L)/\varpi^{n+1}(\mathbf L)) &\to
  \pi_0(\U(\mathbf L)/\varpi^{n+1} (\mathbf L))\to \pi_0(\U(\mathbf
  L)/\varpi^n(\mathbf L))\to 0
\end{alignat*}
and a similar computation as above yields
\begin{equation}\label{eq:bar E^1}
  \overline{E}^1_{n,0}=\mathfrak D_0\overline\gr_n L=\frac{\varpi^n(F)}{\varpi^{n+1}(F)+\rr(n-1)}\text{ for all }n\geq 2.
\end{equation}

\subsection{The linear part of the dimension problem}
Let $\mathbf L$ be a free resolution of $L$, and write $\mathbf
X=\mathbf L_{\text{ab}}=\Gamma_2(\mathbf L)$, which we view as a
graded module concentrated in degree $1$. We recall the ``free
associative algebra'', ``free Lie algebra'' and ``universal envelope''
functors $\mathscr T,\mathscr L,\U$ respectively, with $\mathscr
T=\mathscr U\circ\mathscr L$. If $\mathbf X$ be a graded module, then
$\mathscr T(\mathbf X)$ and $\mathscr L(\mathbf X)$ are naturally
graded.

\begin{prop}[Essentially~\cite{Gruenenfelder:80}*{Theorem~3.3}]\label{prop:iota injective}\verb+ +
  \begin{enumerate}\def\theenumi{\roman{enumi}}
  \item $E^1_{m,n}=\pi_m(\mathscr L_n(\mathbf X))$ and $\overline
    E^1_{m,n}=\pi_m(\mathscr T_n(\mathbf X))$;
  \item $\iota^1_{n,0}:E^1_{n,0}\to\overline E^1_{n,0}$ is injective.
  \end{enumerate}
\end{prop}
\begin{proof}
  Since $\mathbf L$ is free, we have $\mathbf L=\mathscr L(\mathbf X)$
  and $\U(\mathbf L)=\mathscr T(\mathbf X)$; so (i) follows
  immediately from (1) follows from~\eqref{eq:definition E^1}
  and~\eqref{eq:definition bar E^1}.  Now by~(i) we have
  \[E^1_{n,0}=\pi_0(\mathscr L_n(\mathbf X))=\mathscr L_n(\pi_0(\mathbf X))
  \]
  because `$\mathscr L$' preserves coequalizers so commutes with
  $\pi_0$, see Proposition~\ref{prop:coequalizer}, and similarly, because
  `$\mathscr T$' and `$\U$' preserve coequalizers,
  \[\overline E^1_{n,0}=\pi_0(\mathscr T_n(\mathbf X))=\U(\pi_0(\mathscr L(\mathbf X)))=\mathscr T_n(\pi_0(\mathbf X)).
  \]
  The map $\iota^1_{n,0}$ is the degree-$n$ part of the natural map
  $\iota:\mathscr L(\pi_0(\mathbf X))\to\mathscr T(\pi_0(\mathbf X))$,
  which is injective by Theorem~\ref{thm:fund}, so~(ii) is proven.
\end{proof}

We deduce an analogue of one of Sjogren's main results
from~\cite{Sjogren:79}. This is at the time of writing the main
outcome of the simplicial approach, and should serve as a
stepping-stone for further investigation of dimension quotients of Lie
rings:
\begin{theorem}\label{thm:R(n-1)}
  For every free Lie ring $F$ over $\Z$, and all $n\geq 1$, we have
  \[F\cap (\varpi^{n+1}(F)+\rr(n-1))=\gamma_{n+1}(F)+R(n-1).\]
\end{theorem}
\begin{proof}
  The statement is equivalent to the claim
  \[\frac{\gamma_{n+1}(F)+R(n-1)}{\gamma_{n+1}(F)}=\frac{\gamma_n(F)\cap(\varpi^{n+1}(F)+\rr(n-1))}{\gamma_{n+1}},\]
  an equality that takes place in $\gr_n(F)$, or (by
  Theorem~\ref{thm:fund}) equivalently in $\overline\gr_n(F)$, via the
  embedding $\iota:F\mapsto \U(F)$. The statement is therefore
  equivalent, using~\eqref{eq:E^1} and~\eqref{eq:bar E^1}, to
  injectivity of $\iota^1_{n,0}$.
\end{proof}
  
\subsection{Consequences for the dimension problem}
Note that the special case $n=2$ of Theorem~\ref{thm:R(n-1)} already
yields the identification of the third Lie dimension subring, namely
$\delta_3(L)=\gamma_3(L)$. We now show how this result is naturally
interpreted using the two spectral sequences from~\S\ref{ss:ss}, and
how they bear on the identification of $\delta_n(L)/\gamma_n(L)$.

Thanks to our computations~(\ref{eq:E^1},\ref{eq:bar E^1}) for
$E^1_{*,0},\overline E^1_{*,0}$ and~(\ref{eq:D_1},\ref{eq:bar D_1})
for $E^1_{1,1},\overline E^1_{1,1}$ using $\gr_1=\varGamma_2$ and
$\overline\gr_1=\overline\varGamma_2$, we obtain the beginning of the
first pages of the spectral sequences:
\[\begin{tikzpicture}[commutative diagrams/every diagram]
  \matrix[matrix of math nodes, name=m] {
    \frac{\gamma_2\cap R}{[R,F]}\\
    \frac F{\gamma_2+R} & \frac{\gamma_2}{\gamma_3+[R,F]}\\};
  \path[commutative diagrams/.cd, every arrow, every label]
  (m-1-1) edge (m-2-2);
  \node (x1) [xshift=-5mm,yshift=1mm] at (m-1-1.north) {};
  \node (x2) [xshift=-5mm,yshift=-4mm] at (m-2-1.center) {};
  \node (x3) [xshift=1mm,yshift=-4mm] at (m-2-2.east) {};
  \node at (m-1-1) [xshift=-9mm] {$\scriptscriptstyle q=1$};
  \node at (m-2-1) [xshift=-9mm] {$\scriptscriptstyle q=0$};
  \node at (m-2-1) [yshift=-6mm] {$\scriptscriptstyle p=1$};
  \node at (m-2-2) [yshift=-6mm] {$\scriptscriptstyle p=2$};
  \node at (m-2-1) [yshift=-5mm,xshift=-8mm] {$E^1_{p,q}$};
  \draw (x1) -- (x2.center) -- (x3);
\end{tikzpicture}\text{ and }
\begin{tikzpicture}[commutative diagrams/every diagram]
  \matrix[matrix of math nodes, name=m] {
    \frac{\varpi^2\cap\rr}{[R,F]}\\
    \frac F{\varpi^2+\rr} & \frac{\varpi^2}{\varpi^3+\rr\varpi+\varpi\rr}\\};
  \path[commutative diagrams/.cd, every arrow, every label]
  (m-1-1) edge (m-2-2);
  \node (x1) [xshift=-5mm,yshift=1mm] at (m-1-1.north) {};
  \node (x2) [xshift=-5mm,yshift=-4mm] at (m-2-1.center) {};
  \node (x3) [xshift=1mm,yshift=-4mm] at (m-2-2.east) {};
  \node at (m-1-1) [xshift=-9mm] {$\scriptscriptstyle q=1$};
  \node at (m-2-1) [xshift=-9mm] {$\scriptscriptstyle q=0$};
  \node at (m-2-1) [yshift=-6mm] {$\scriptscriptstyle p=1$};
  \node at (m-2-2) [yshift=-6mm] {$\scriptscriptstyle p=2$};
  \node at (m-2-1) [yshift=-5mm,xshift=-8mm] {$\overline E^1_{p,q}$};
  \draw (x1) -- (x2.center) -- (x3);
\end{tikzpicture}\]

On the last page, we then see
\[\begin{tikzpicture}[commutative diagrams/every diagram]
  \matrix[matrix of math nodes, name=m] {
    \frac{(\gamma_3+[R,F])\cap R}{[R,F]}\\
    \frac F{\gamma_2+R} & \frac{\gamma_2}{\gamma_3+(\gamma_2\cap R)}\\};
  \path[commutative diagrams/.cd, every arrow, every label]
  (m-1-1) edge (m-2-2);
  \node (x1) [xshift=-10mm,yshift=1mm] at (m-1-1.north) {};
  \node (x2) [xshift=-10mm,yshift=-4mm] at (m-2-1.center) {};
  \node (x3) [xshift=1mm,yshift=-4mm] at (m-2-2.east) {};
  \node at (m-1-1) [xshift=-14mm] {$\scriptscriptstyle q=1$};
  \node at (m-2-1) [xshift=-14mm] {$\scriptscriptstyle q=0$};
  \node at (m-2-1) [yshift=-6mm] {$\scriptscriptstyle p=1$};
  \node at (m-2-2) [yshift=-6mm] {$\scriptscriptstyle p=2$};
  \node at (m-2-1) [yshift=-5mm,xshift=-13mm] {$E^\infty_{p,q}$};
  \draw (x1) -- (x2.center) -- (x3);
\end{tikzpicture}\text{ and }
\begin{tikzpicture}[commutative diagrams/every diagram]
  \matrix[matrix of math nodes, name=m] {
    \frac{(\varpi^3+\varpi\rr+\rr\varpi)\cap\rr}{\varpi\rr+\rr\varpi}\\
    \frac F{\varpi^2+\rr} & \frac{\varpi^2}{\varpi^3+(\varpi^2\cap\rr)}\\};
  \path[commutative diagrams/.cd, every arrow, every label]
  (m-1-1) edge (m-2-2);
  \node (x1) [xshift=-12mm,yshift=1mm] at (m-1-1.north) {};
  \node (x2) [xshift=-12mm,yshift=-4mm] at (m-2-1.center) {};
  \node (x3) [xshift=1mm,yshift=-4mm] at (m-2-2.east) {};
  \node at (m-1-1) [xshift=-16mm] {$\scriptscriptstyle q=1$};
  \node at (m-2-1) [xshift=-16mm] {$\scriptscriptstyle q=0$};
  \node at (m-2-1) [yshift=-6mm] {$\scriptscriptstyle p=1$};
  \node at (m-2-2) [yshift=-6mm] {$\scriptscriptstyle p=2$};
  \node at (m-2-1) [yshift=-5mm,xshift=-15mm] {$\overline E^\infty_{p,q}$};
  \draw (x1) -- (x2.center) -- (x3);
\end{tikzpicture}\]

\noindent In particular,
\[E^2_{2,0}=\gamma_2(L)/\gamma_3(L)\text{ and }\overline E^2_{2,0}=\varpi^2(L)/\varpi^3(L).\]

The morphism $\boldsymbol\iota$ induced on spectral sequences by the
inclusion $L\to\U(L)$ yields then the following commutative diagram
with exact rows:
\[\begin{tikzcd}
  E^{1}_{1,1}\arrow{r}{d^1}\arrow{d}{\iota^1_{1,1}} & E^1_{2,0}\arrow{r}\arrow{d}{\iota^1_{2,0}} & E^2_{2,0}\arrow{r}\arrow{d}{\iota^2_{2,0}} & 0\\
  \overline{E}^{1}_{1,1}\arrow{r}{\overline d^1} & \overline{E}^1_{2,0}\arrow{r} & \overline{E}^2_{2,0}\arrow{r} & 0.
\end{tikzcd}
\]
The homomorphism $\iota^1_{1,1}$ is an isomorphism because the
dimension problem is true for abelian Lie algebras. The homomorphism
$\iota^1_{2,0}$ is injective by Proposition~\ref{prop:iota
  injective}(ii). It follows therefore that $\iota^2_{2,0}$ is
injective. We deduce a new proof of Theorem~\ref{thm:delta3}:
\begin{quote}
  \itshape For every Lie algebra over $\Z$, we have
  $\delta_3(L):=L\cap\varpi^3(L)=\gamma_3(L)$.
\end{quote}

Similarly, $\delta_4(L)/\gamma_4(L)$ is an epimorphic image of
$\ker(\iota^2_{3,0}:E^2_{3,0}\to \overline{E}^2_{3,0})$; and, more
generally, Propositions~\ref{prop:E^infty} and~\ref{prop:bar E^infty}
imply the
\begin{theorem} For every Lie ring $L$,
  \[\frac{\gamma_n(L)\cap\delta_{n+1}(L)}{\gamma_{n+1}(L)}=\ker\left[\iota^\infty_{n,0}\colon E^\infty_{n,0}\to\overline E^\infty_{n,0}\right].\qed\]
\end{theorem}

\section{Concluding remarks}
We have attempted to show, in this text, how the theories of groups
and Lie rings parallel each other; and, in particular, how tools
developed to study dimension quotients in one context bear on the
other.

It is our general feeling that most aspects become simpler when
transposed to the Lie ring setting. The only exception we are aware of
is Fox's problem, see Example~\ref{example:fox}.

We believe that, at least in the Lie algebra setting, an optimal bound
on the exponent of the dimension quotients, sharpening Sjogren's
bound, should be achievable.

One particularly tempting direction for further investigations is that
of metabelian Lie rings. What is the best that one can say about
dimension quotients in that case?

In the corresponding group case, the two series agree from some stage
onward, depending only on $G/[G, G]$, see~\cite{Gupta-Hales-Passi:84}.
Gupta also has obtained in~\cites{Gupta:84,Gupta:91} a bound for the
exponent of the dimension quotient that is much smaller than the
Sjogren bound from~\cite{Sjogren:79}.

Our Theorem~\ref{thm:deltan} shows much more: in the metabelian case,
the exponent of the $n$th dimension quotient is bounded by the
exponent of the torsion of $L/[L,L]$.

\section*{Acknowledgment}
The second author would like to express sincere thanks for warm
hospitality provided by the Mathematisches Institut der Georg-August
Universit\"at, G\"ottingen, where the work on this paper was first
taken up during his visit in 2012.


\end{document}